\documentclass[12pt, a4paper]{article}
\usepackage{latexsym, amsxtra, amsmath, amssymb, amsthm, amscd}

\newtheorem{theorem}{Theorem}[section]
\newtheorem{lem}[theorem]{Lemma}
\newtheorem{cor}[theorem]{Corollary}
\newtheorem{pro}[theorem]{Proposition}
\theoremstyle{remark}
\newtheorem{rem}[theorem]{Remark}
\newtheorem{ex}[theorem]{Example}

\newtheorem{de}[theorem]{Definition}

\newcommand{\un}{\ensuremath{\underline}}
\def\D{\mathcal{D}}
\def\F{\mathcal{F}}
\def\Ass{\mathrm{Ass}}
\def\Ann{\mathrm{Ann}}
\def\Spec{\mathrm{Spec}}
\def\Supp{\mathrm{Supp}}
\def\Assh{\mathrm{Assh}}
\def\fm{\frak{m}}
\def\fn{\frak{n}}
\def\fp{\frak{p}}
\def\fq{\frak{q}}
\def\I{I_{\F, M}}

\begin{document}

\title{On a new invariant of finitely generated  modules\\ over local rings}        
\author{Nguyen Tu Cuong\footnote{Email: ntcuong@math.ac.vn}\ , Doan Trung Cuong\footnote{Email: dtcuong@math.ac.vn}\ \  and Hoang Le Truong\footnote{Email: hltruong@math.ac.vn}\\
Institute of Mathematics\\
18 Hoang Quoc Viet Road, 10307 Hanoi, Vietnam}     
\date{ }
\maketitle

\begin{abstract}
Let $M$ be a finitely generated module on a local ring $R$ and $\F: M_0\subset M_1\subset\ldots\subset M_t=M$ a filtration  of submodules of $M$ such that $ d_o<d_1< \ldots <d_t=d$, where $d_i=\dim M_i$.  This paper is concerned with a non-negative integer $p_\mathcal F(M)$  which is defined  as  the least degree of all polynomials in $n_1,\ldots , n_d$  bounding above the function
$$\ell(M/(x_1^{n_1}, \ldots, x_d^{n_d})M)-\sum_{i=0}^tn_1\ldots n_{d_i}e(x_1,\ldots, x_{d_i};M_i).$$
We prove that $p_\mathcal F(M)$
is independent of the choices of good systems of parameters $\underline x=x_1, \ldots, x_d$.  When $\F$ is the dimension filtration of $M$ we also present some relations between $p_\F(M)$ and the polynomial type of each $M_i/M_{i-1}$ and the dimension of the non-sequentially Cohen-Macaulay locus of $M$.\\
{\it Key words:} multiplicity, dimension filtration, filtration satisfies the dimension condition, good system of parameters.\\
{\it AMS Classification:} 13H15, 13H10, 13C15.

\end{abstract}

\section{Introduction}
Let $(R, \mathfrak m)$ be a commutative Noetherian local ring and $M$ a finitely generated $R$-module of dimension $d$. We consider a finite filtration  $\F: M_0\subset M_1\subset\ldots\subset M_t=M$ of submodules of $M$ such that $\dim M_0<\dim M_1<\ldots < \dim M_t=\dim M$. Such a filtration is said to satisfy the dimension condition. Let $\un x=x_1, \ldots, x_d$ be a system of parameters of $M$. Then $\underline x$ is called  a good system of parameters with respect to $\mathcal F$ if $M_i\cap (x_{d_i+1},\ldots, x_d)M=0$ for $i=0, 1, \ldots, t-1,$ where $d_i=\dim M_i$. Set 
$$\I(\underline x)=\ell(M/\underline xM)-\sum_{i=0}^te(x_1,\ldots, x_{d_i};M_i),$$
where $e(x_1,\ldots, x_{d_i};M_i)$ is the Serre multiplicity of $M_i$ with respect to $x_1,\ldots, x_{d_i}$. Denote $\underline x(\underline n)=x_1^{n_1},\ldots,x_d^{n_d}$ for any $d$-tuple of positive integers $n_1,\ldots,n_d$. It is shown in \cite{CC1} that $\I(\underline  x)$ is non-negative and $\I(\underline  x(\underline  n))$ is non-decreasing as a function in $n_1,\ldots, n_d$. In fact, $\I(\underline x(\underline  n))$ is not a polynomial in $n_1, \ldots, n_d$ in general. However, it can be seen easily that this function is bounded above by a polynomial. In this paper we study the least degree of the polynomials bounding above $\I(\underline  x(\underline  n))$ and show that this degree is independent of the choices of good systems of parameters with respect to $\mathcal F$. This is the content of the following theorem.

\begin{theorem}\label{A} Let $\F: M_0\subset M_1\subset\ldots\subset M_t=M$ be a filtration of submodules of $M$ satisfying the dimension condition and $\un x=x_1, \ldots, x_d$ a good system of parameters with respect to $\mathcal F$. Then the least degree of all polynomials in $n_1,\ldots,n_d$ bounding above the function $\I(\underline x(\underline  n))$ is independent of the choice of  $\underline x$. 
\end{theorem}
The least degree mentioned in Theorem \ref{A} is denoted by $p_\mathcal F(M)$. Let $\mathcal F_0 : 0\subset M$ be  the trivial filtration of $M$. Then every system of parameters of $M$ is good with respect to $\mathcal F_0$. By applying Theorem \ref{A} in this case we get again one of the main results of \cite{C2}. Note that in that paper the invariant $p_{\mathcal F_0}(M)$ was called the polynomial type of $M$ and denoted by $p(M)$. Under some mild assumptions, there are very closed relations between the polynomial type $p(M)$ and the annihilators of the local cohomology modules and the dimension of the non-Cohen-Macaulay locus of $M$. In general it is natural to question how to relate $p_\mathcal F(M)$ with other known invariants of $M$. In the present paper we have not yes had a general answer to this question. When the filtration is the dimension filtration of $M$, i. e., a filtration $\D: D_0\subset D_1\subset\ldots\subset D_t=M$  where $D_i$ is the biggest submodule of $D_{i+1}$ such that $\dim D_i<\dim D_{i+1}$, we give an explicite relation of $p_\mathcal D(M)$ with the polynomial type $p(D_i/D_{i-1})$ and the dimension of the non-sequentially Cohen-Macaulay locus $V_M$ of $M$ provided $R$ is a quotient of a Cohen-Macaulay ring. 
 \begin{theorem}\label{B} Let $R$ be a quotient of a Cohen-Macaulay ring and $\D: D_0\subset D_1\subset\ldots\subset D_t=M$ the dimension filtration of $M$. Then we have 
$$p_\mathcal D(M)=\max\{p(D_i/D_{i-1}): i=1, 2, \ldots, t\}=\dim V_M.$$ 
\end{theorem}
The paper is divided into six sections. Some preliminaries on filtrations satisfying the dimension condition and good systems of parameters are given in Section 2. Section 3 is devoted  to proving Theorem \ref{A}. In Section 4 we investigate the behavior of the dimension filtration and good systems  of parameters under localization. These results are used to relate $p_\mathcal D(M)$ with the invariants $p(D_i/D_{i-1}), i=1, \ldots, t$ and the dimension of the non-sequentially Cohen-Macaulay locus of $M$ in Section 5 where we prove Theorem \ref{B}. The behavior of the invariant $p_\F(M)$ under flat extensions is studied in the last section.


\section{Preliminaries}
Throughout this note, let $(R, \frak m)$ be a commutative Noetherian local ring and $M$ a finitely generated $R$-module of dimension $d$. In \cite{CC1}, \cite{CC2} the notions of filtrations satisfying the dimension condition and good system of parameters have been introduced to study the structure of sequentially Cohen-Macaulay and sequentially generalized Cohen-Macaulay modules. In this section we recall briefly the definitions and some first properties of these notion those are used in the remains of this note. For more details we refer to the papers \cite{CC1}, \cite{CC2}.
\begin{de}
(i) We say that a finite filtration of submodules of $M$
$$\mathcal F : \ M_0\subset M_1\subset \cdots \subset M_t=M$$
 satisfies the {\it  dimension condition} if
$\dim M_0<\dim M_1<\ldots <\dim M_{t-1}<\dim M$, where we stipulate that $\dim M=-\infty$ if $M=0$.
\item(ii) A filtration $\D : D_0\subset D_1\subset \cdots \subset D_t=M$ is called the {\it dimension filtration} of $M$ if the following two conditions are satisfied

a) $D_{i-1}$ is the largest submodule of $D_i$ with $\dim D_{i-1}< \dim D_i$ for $i= t, t-1, \ldots, 1$;

b) $D_0=H_\fm^0(M)$ is the $0^{\text{th}}$ local cohomology module of $M$ with respect to the maximal ideal $\fm$.
\item(iii) Let $\F: M_0\subset M_1\subset\ldots\subset M_t=M$ be a filtration satisfying the dimension condition with $d_i=\dim M_i$. A system of parameters $\underline x=x_1, \ldots, x_d$ of $M$ is called a good system of parameters with respect to the filtration $\mathcal{F}$ if $M_i\bigcap(x_{d_i+1},\ldots,x_d)M=0$ for $i=0,1,\ldots,t-1$. A good system of parameters with respect to the dimension filtration is simply called a {\it good system of parameters} of $M$.
\end{de}
\begin{rem}\label{23} 
 The dimension filtration always exists and it is unique, in this note we will denote the dimension filtration by 
$$\mathcal D : D_0\subset D_1\subset\ldots\subset D_t=M.$$
In fact, let $0=\bigcap\limits_{\frak{p}\in\mathrm{Ass}(M)}N(\frak{p})$ be a reduced primary decomposition in $M$, then $D_i=\bigcap\limits_{\dim R/\frak{p}\ge\dim D_{i+1}}N(\frak{p})$.
\end{rem}
\begin{rem}
 Good systems of parameters always exist (see \cite{CC1}, Lemma 2.5). Using the notations as above, denote $d_i=\dim D_i$ and $N_i=\bigcap\limits_{\dim R/\frak{p}\leqslant \dim D_i}N(\frak{p})$. Then $D_i\cap N_i=0$. By the Prime Avoidance Theorem, there is a system of parameters $\underline x=x_1, \ldots, x_d$ of $M$ such that $x_{d_i+1}, \ldots, x_d\in \Ann(M/N_i)$, $i=0, \ldots, t$. Hence $(x_{d_i+1}, \ldots, x_d)M\cap D_i=0$ and $\un x$ is a good system of parameters of $M$. It should be noted that in the last part of this note we sometimes use this idea to show the existence of good systems of parameters of some modules which satisfy some expected properties.
 \end{rem}
\begin{rem}
Let $N$ be a submodule of $M$. From the definition of the dimension filtration there exists a module $D_i$ such that $N\subseteq D_i$ and $\dim N=\dim D_i$. Consequently, if $M_0\subset M_1\subset\ldots\subset M_{t'}=M$ is a filtration satisfying the dimension condition then there are $0\leqslant i_0< i_1<\ldots<i_{t'}$ such that $M_j\subseteq D_{i_j}$ and $\dim M_j=\dim D_{i_j}$. Thus a good system of parameters is also a good system of parameters with respect to any filtration satisfying the dimension condition. Therefore good systems of parameters with respect to a filtration satisfying the dimension condition always exist by the previous argument.
\end{rem}


\section{Proof of Theorem \ref{A}}
In this section we always denote by $\F: M_0\subset M_1\subset \ldots \subset M_t=M$ a filtration of submodules of $M$ satisfying the dimension condition and $\un x=x_1, \ldots, x_d$  a good system of parameters of $M$ with respect to $\F$. Put $d_i=\dim M_i$. It is easy to see that $x_1, \ldots, x_{d_i}$ is a system of parameters of $M_i$ for $i=1, \ldots, t$. So the following difference is well defined
$$I_{\F, M}(\un x)=\ell(M/\un xM)-\sum_{i=0}^te(x_1, \ldots, x_{d_i}, M_i),$$
where $e(x_1,\ldots,x_{d_i};M_i)$ is the Serre multiplicity and we set $e(x_1,\ldots,x_{d_0};M_0)=\ell(M_0)$ if $M_0$ is of finite length. Note that for the case $\F$ is the trivial filtration $0\subset M$, $I_{\F, M}(\un x)$ is just the difference $I_M(\un x)=\ell(M/\un xM)-e(\un x, M)$ which is well known in the theory of Buchsbaum rings. In the general case, it is proved in \cite[Lemma 2.7]{CC1} that $I_{\F, M}(\un x)$ is non-negative. Therefore,  this is a generalization of the well known inequality $\ell(M/\un xM)\geqslant e(\un x, M)$ between the length function and multiplicity of a system of parameters . Moreover, set $\un x(\un n)=x_1^{n_1}, \ldots, x_d^{n_d}$ for any $d$-tuple of positive integers $\un n =n_1, \ldots, n_d$, we can consider $I_{\F, M}(\un x(\un n))$ as a function in $n_1, \ldots, n_d$. By \cite[Proposition 2.9]{CC1}, $I_{\F, M}(\un x(\un n))$ is non-decreasing, i. e., $I_{\F, M}(\un x(\un n))\leqslant I_{\F, M}(\un x(\un m))$ for all $n_i\leqslant m_i$, $i=1, \ldots, d$. The function $I_{\F, M}(\un x(\un n))$ has been first considered in \cite{CC1}, \cite{CC2} and is a useful tool in the study of the structure of sequentially Cohen-Macaulay and sequentially generalized Cohen-Macaulay modules.
We begin with some lemmas which is needed in the proof of Theorem \ref{A}.
\begin{lem}\label{noisop}
Let $\F: M_0\subset M_1\subset\ldots\subset M_t=M$ be a filtration satisfying the dimension condition and  $\underline x=x_1, \ldots, x_d$, $\underline y=y_1, \ldots, y_d$  good systems of parameters of $M$ with respect to $\mathcal F$. Then there exist  a good system of parameters $\underline z=z_1, \ldots, z_d$ of $M$ with respect to $\mathcal F$ and positive integers $r_1,\ldots,r_d, s$ such that $z_1\ldots,z_i,x_{i+1}^{r_{i+1}},\ldots,x_d^{r_d}$ and $z_1\ldots,z_i,y_{i+1}^s,\ldots,y_d^{s}$ are  good systems of parameters of $M$ with respect to $\mathcal F$ for $i=1,\ldots , d$ and
$$(x_1^{r_1},\ldots,x_d^{r_d})+\mathrm{Ann}M\subseteq(z_1,x_2^{r_2}\ldots,x_d^{r_d})+\mathrm{Ann}M\subseteq\ldots\subseteq(\underline{z})+\mathrm{Ann}M$$
$$\subseteq(z_1,\ldots,z_{d-1},y_d^s)+\mathrm{Ann}M\subseteq\ldots\subseteq(y_1^{s},\ldots,y_d^{s})+\mathrm{Ann}M.$$
\end{lem}
\begin{proof} Let $\bigcap\limits_{\frak{p}\in\mathrm{Ass}(M)}N(\frak{p})=0$ be a reduced primary decomposition of the zero submodule of $M$, where $N(\frak{p})$ is a $\frak{p}$-primary submodule.  Since 
$$M_i/(M_i\cap N(\frak{p}))\cong (M_i+N(\frak{p}))/N(\frak{p})\subset M/N(\frak{p}),$$   it follows that $M_i\not\subseteq N(\frak{p})$ if and only if  $\mathrm{Ass}( M_i/M_i\bigcap N(\frak{p}))=\{\frak{p}\}$, i. e., $M_i\bigcap N(\frak{p})$ is a $\frak{p}$-primary submodule of $M_i$. Hence $M_i\subseteq N(\frak{p})$, if $\frak{p}\not\in\mathrm{Ass}M_i$. Therefore  $M_i\subseteq \bigcap\limits_{\frak{p}\not\in\mathrm{Ass}M_i}N(\frak{p})$.  This implies that $\sqrt{\mathrm{Ann}M_i}=\sqrt{\mathrm{Ann}M/N_i}$, where $N_i=\bigcap\limits_{\frak{p}\in\mathrm{Ass}M_i}N(\frak{p})$. Set $d_i=\dim M_i$.  Since $\underline x$ and $\underline y$ are  good systems of parameters of $M$ with respect to $\mathcal F$, we have $(y_{d_i+1},\ldots,y_d)\subseteq \mathrm{Ann}M_i$. Hence there is a positive integer $s$ such that $(y_{d_i+1}^s,\ldots,y_d^s)\subseteq \mathrm{Ann}M/N_i$ for all $i=1,\ldots,t$. Similarly, replace $x_i$ with  $x_i^n$ for large enough $n$, we can assume without any loss of generality that $(x_{d_i+1},\ldots,x_d)\subseteq \mathrm{Ann}M/N_i$ for all $i=1,\ldots, t$.

Next, we claim that there is a good system of parameters $\underline{w}=w_1, \ldots, w_d$ with respect to $\mathcal{F}$ such that for all $i\in \{0, 1,\ldots, d\}$, $w_1,\ldots,w_i,x_{i+1},\ldots,x_d$ and $w_1,\ldots,w_i$, $y_{i+1}^s,\ldots,y_d^s$ are good systems of parameters of $M$ with respect to $\mathcal{F}$. In fact, the case $i=0$ is trivial. Assume that $i>0$ and we have chosen $w_1,\ldots,w_{i-1}$. Let $S$ be the set of all minimal associated prime ideals of $M/(w_1,\ldots,w_{i-1},x_{i+1},\ldots,x_d)M$ and of $M/(w_1, \ldots, w_{i-1}, y_{i+1}^s, \ldots, y_d^s)M$. Since $w_1,\ldots, w_{i-1}, x_i, \ldots, x_d$ and $w_1,\ldots,w_{i-1},y_i^s,\ldots,y_d^s$ are system of parameters of $M$,  $(x_i, y_i^s)\not\subseteq\bigcup\limits_{\fp\in S}\fp$. Assume that $d_j< i\leqslant d_{j+1}$ for some $j$. 
Keep in mind  that $(x_i, y_i^s)\subseteq \bigcap\limits_{k\leqslant j}\mathrm{Ann}(M/N_k)$. Therefore we can choose an element $w_i \in \bigcap\limits_{k\leqslant j}\mathrm{Ann}(M/N_k)$ such that $w_i\not\in\frak{p}$ for any $\frak{p}\in S$. It follows that $w_1,\ldots,w_i,x_{i+1},\ldots,x_d$, $w_1,\ldots,w_i,y_{i+1}^s,\ldots,y_d^s$ are systems of parameters of $M$. By the first part of the proof and the choice of $w_i$ we have for all $k\leqslant j,$
$$(w_{d_k+1},\ldots,w_i,x_{i+1},\ldots,x_d)M\subseteq N_k,$$
$$(w_{d_k+1},\ldots,w_i,y_{i+1}^s,\ldots,y_d^s)M\subseteq N_k.$$
Therefore
$$(w_{d_k+1},\ldots,w_i,x_{i+1},\ldots,x_d)M\cap M_k=(w_{d_k+1},\ldots,w_i,y_{i+1}^s,\ldots,y_d^s)M\cap M_k=0.$$
Thus $w_1,\ldots,w_i,x_{i+1},\ldots,x_d$, $w_1,\ldots,w_i,y_{i+1}^s,\ldots,y_d^s$ are good system of parameters of $M$ with respect to  $\mathcal{F}$ and the claim is proved.
Now, by inductive method  there are positive integers $t_1,\ldots , t_d$ such that 
$$w_i^{t_i} \in (w_1^{t_1}, \ldots, w_{i-1}^{t_{i-1}}, y_i^s,\ldots,y_d^s)+\mathrm{Ann}M.$$
Set $z_i=w_i^{t_i}$. Then it is easy to check that  $z_1, \ldots, z_d$ is the required systems of parameters, where  $r_1, \ldots, r_d$ can be chosen so  that 
$$x_i^{r_i}\in (z_1,\ldots,z_{i},x_{i+1}^{r_{i+1}},\ldots,x_d^{r_d})+\mathrm{Ann}M.$$
\end{proof}

\begin{lem}\label{inequal}
Let $\underline{x}$, $\underline{y}$ be two good systems of parameters of $M$ with respect to $\mathcal{F}$. Assume  that there exists a positive integer $i$ such that $x_j=y_j$ for all $j\not=i$ and $(\underline{y})+\mathrm{Ann}M\subseteq(\underline{x})+\mathrm{Ann}M$. Then $I_{\mathcal{F},M}(\underline{x})\leqslant I_{\mathcal{F},M}(\underline{y})$.
\end{lem}
\begin{proof}We prove the lemma by induction on  $d=\dim M$. The case $d=1$ is obvious. Assume $d>1$. Since  $\ell(M/(\underline{x})M)-e(\underline{x};M)\leqslant\ell(M/(\underline{y})M)-e(\underline{y};M)$ by \cite{C2},  the case $i=d$ is proved. Assume now that $i<d$.  From Lemma 2.7 of \cite{CC1}, $x_1,\ldots, x_{d-1}$ is a good system of parameters of $M/x_dM$ with respect to the  filtration satisfying the dimension condition
$$\mathcal{F}/x_d\mathcal{F}:(M_0+x_dM)/x_dM\subset\ldots\subset (M_s+x_dM)/x_dM\subset M/x_dM,$$ 
where $s=t-1$ if $d_{t-1}<d-1$ and $s=t-2$ if $d_{t-1}=d-1$,  and $$I_{\mathcal{F},M}(\underline{x})=I_{\mathcal{F}/x_d\mathcal{F},M/x_dM}(x_1,\ldots,x_{d-1})+e(x_1,\ldots,x_{d-1};0:_Mx_d/M_{t-1}),$$ 
$$I_{\mathcal{F},M}(\underline{y})=I_{\mathcal{F}/x_d\mathcal{F},M/x_dM}(y_1,\ldots,y_{d-1})+e(y_1,\ldots,y_{d-1};0:_Mx_d/M_{t-1}).$$
Therefore $I_{\mathcal{F},M}(\underline{x})\leqslant I_{\mathcal{F},M}(\underline{y})$ by the induction hypothesis and the fact that $$e(x_1,\ldots,x_{d-1};0:_Mx_d/M_{t-1})\leqslant e(y_1,\ldots,y_{d-1};0:_Mx_d/M_{t-1}).$$ 
 \end{proof}

\begin{lem}\label{poly}
Let $\mathcal{F}$ be a filtration satisfying the dimension condition and $\underline{x}=x_1, \ldots, x_d$ a good system of parameters with respect to $\mathcal{F}$. Assume that $I_{\mathcal{F},M}(\underline{x})\not=0$. Then there is a constant $c$ such that $I_{\mathcal{F},M}(\underline x(\underline n))\leqslant cn_1\ldots n_dI_{\mathcal{F},M}(\underline{x})$, for all $n_1,\ldots, n_d>0$. In particular, $I_{\mathcal{F},M}(\underline x(\underline n))$ is bounded above by a polynomial in $n_1, \ldots, n_d$.
\end{lem}
\begin{proof} By \cite[Lemma 2.1]{C2} we have 
$$\ell(M/(\underline x(\underline n))M)-n_1\ldots n_de(\underline{x};M)\leqslant n_1\ldots n_d\big(\ell(M/(\underline{x})M)-e(\underline{x};M)\big).$$
Hence 
$$I_{\mathcal{F},M}(\underline x(\underline n))\leqslant n_1\ldots n_d\big(\ell(M/(\underline{x})M)-e(\underline{x};M)\big)\leqslant cn_1\ldots n_dI_{\mathcal{F},M}(\underline{x}),$$ 
where $c=\big(\ell(M/(\underline{x})M)-e(\underline{x};M)\big)/I_{\mathcal{F},M}(\underline{x})$. 
\end{proof}
Note that without the assumption $I_{\mathcal{F},M}(\underline{x})\not=0$ the lemma  is no longer true. The examples can be found in   \cite [Example 4.7]{CC1}.

\begin{cor}\label{bdt}
Let  $\underline{x},\:\underline{y}$ be two good systems of parameters of $M$ with respect to $\mathcal{F}$ as in Lemma \ref{inequal}. Assume in addition that $I_{\mathcal{F},M}(\underline{y})\not=0$.  Then there is a constant $c$ such that
$$I_{\mathcal{F},M}(x_1^t,\ldots,x_d^t)\leqslant cI_{\mathcal{F},M}(x_1^t,\ldots, x_{i-1}^t, y_i^{t}, x_{i+1}^t,\ldots,x_d^t)$$
for all positive integers $t$.
\end{cor}
\begin{proof} It is straightforward by Lemmas \ref{inequal} and \ref{poly}. 
\end{proof}
 
\noindent{\bf Proof of Theorem \ref{A}.} Let $\underline{x}$ and $\underline{y}$ be two good systems of parameters of $M$ with respect to $\mathcal{F}$. Denote the least degree of the polynomials bounding  above the functions $I_{\mathcal{F},M}(\underline x(\underline n))$ and $I_{\mathcal{F},M}(\underline{y}(\underline n))$ by $p_{\mathcal{F},M}(\underline{x})$ and $p_{\mathcal{F},M}(\underline{y}),$ respectively. It is enough to prove that $p_{\mathcal{F},M}(\underline{x})\ge p_{\mathcal{F},M}(\underline{y})$.  It is nothing to prove if $I_{\mathcal{F},M}(\underline{y}(\underline{n}))=0$ for all  $n_1,\ldots,n_d$. Therefore, by virtue of the non-decreasing property of the function $I_{\mathcal{F},M}(\underline{y}(\underline n))$,   we can assume in addition that $I_{\mathcal{F},M}(\underline{y}(\underline{n}))\not=0$ for all  $n_1,\ldots,n_d$. 
By Lemma \ref{noisop}, there exist $s, r_1,\ldots, r_d>0$ and a good system of parameters $\underline{z}=z_1,\ldots, z_d$ with respect to $\mathcal{F}$ such that 
$$(x_1^{r_1},\ldots,x_d^{r_d})+\mathrm{Ann}M\subseteq(z_1,x_2^{r_2}\ldots,x_d^{r_d})+\mathrm{Ann}M\subseteq\ldots\subseteq(\underline{z})+\mathrm{Ann}M$$
$$\hspace{1cm}\subseteq(z_1,\ldots,z_{d-1},y_d^s)+\mathrm{Ann}M\subseteq\ldots\subseteq(y_1^{s},\ldots,y_d^{s})+\mathrm{Ann}M.$$
where $z_1, \ldots, z_i, x_{i+1}^{r_{i+1}}, \ldots, x_d^{r_d}$ and $z_1, \ldots, z_i, y_{i+1}^s, \ldots, y_d^{s}$ are good systems of parameters of $M$ with respect to $\mathcal{F}$. Applying Corollary \ref{bdt} $(2d)$-times to the above sequence of systems of parameters, we get
$$I_{\mathcal{F},M}(y_1^{ts},\ldots,y_d^{ts})\leqslant (c)^{2d}I_{\mathcal{F},M}(x_1^{tr_1},\ldots,x_d^{tr_d}),$$ 
for all positive integers $t$,  where $c$ is a constant depending only on the systems of parameters $y_1^s, \ldots, y_d^s$,\  $x_1^{r_1}, \ldots, x_d^{r_d}$ and $\underline{z}$. For  any $d$-tuple $n_1,\ldots,n_d,$ we set $t=n_1+\ldots+n_d$. Then 
$$I_{\mathcal{F},M}(y_1^{n_1},\ldots,y_d^{n_d})\leqslant I_{\mathcal{F},M}(y_1^{st},\ldots,y_i^{st},\ldots,y_d^{st})\leqslant (c)^{2d}I_{\mathcal{F},M}(x_1^{tr_1},\ldots,x_d^{tr_d}).$$
Therefore, if  $F(n_1,\ldots,n_d)$ is a polynomial bounding above $I_{\mathcal{F},M}(x_1^{n_1},\ldots,x_d^{n_d})$, we have 
   $$I_{\mathcal{F},M}(y_1^{n_1},\ldots,y_d^{n_d})\leqslant (c)^{2d}F(r_1n_1+\ldots +r_1n_d,\ldots, r_dn_1+\ldots+ r_dn_d).$$
The last function is  a polynomial in $n_1,\ldots,n_d$ of the same degree as $F(n_1,\ldots,n_d)$. This implies that $p_{\mathcal{F},M}(\underline{x})\ge p_{\mathcal{F},M}(\underline{y})$. 
\hfill$\square$
\medskip

Since every system of parameters of $M$ is good with respect to the trivial filtration $\mathcal F_0 : 0\subset M$, we get the following immediate corollary which is one of the main results in 
\cite[Theorem 2.3]{C2}.
\begin{cor}
Let $\underline x=x_1, \ldots, x_d$ be a system of parameters of $M$ and $\underline n =n_1, \ldots, n_d$ a $d$-tuple of positive integers. Then the least degree of all polynomials in $\underline n$ bounding above the function  $I_M(\un x(\un n))=\ell(M/\un x(\un n)M)-e(\un x(\un n), M)$ is independent of the choice of $\un x$.
\end{cor} 
Recall  from \cite{CN} that an $R$-module $M$ is called a sequentially Cohen-Macaulay module (respectively, generalized Cohen-Macaulay module) if there is a filtration satisfying the dimension condition $\F: M_0\subset M_1\subset \ldots \subset M_t=M$ such that $M_0$ is of finite length and each $M_i/M_{i-1}$ is Cohen-Macaulay (generalized Cohen-Macaulay) for $i=1, \ldots, t$.  By virtue of  \cite [Theorem 4.2]{CC1} and \cite [Theorem 4.6]{CC2}, we get the following  consequence of Theorem \ref{A}. We stipulate here that the degree of the zero polynomial is $-\infty$.
\begin{cor}
$M$ is a sequentially Cohen-Macaulay module (respectively, sequentially generalized Cohen-Macaulay module) if and only if there is  a filtration $\mathcal F$ satisfying the dimension condition such that $p_{\mathcal{F}}(M)= -\infty$ (respectively, $p_{\mathcal{F}}(M)\leqslant 0$).
\end{cor}


\section{Localization}
 In the rest of this paper, we denote  by  $\mathcal{D:}\ D_0\subset D_1\subset\ldots\subset D_t=M$ the dimension filtration of $M$ with $d_i=\dim M_i$. Then, for any  $\frak{p}\in \Supp M$ there exist non negative integers $ s_l,\ldots , s_1$, $d\ge s_l>  \ldots> s_1\ge 0$, which are  determined recursively as follows. $s_l$ is the least integer such that $M_\frak{p}=(D_{s_l})_\frak{p}$ and $s_i$ is the least integer such that  $(D_{s_{i+1}-1})_\frak{p}=(D_{s_i})_\frak{p}$ for all $i=l-1,\ldots,1$. In this section we show that for each localization of a module at a prime, there are good systems of parameters of the module such that a part of it is also a good system of parameters of the localization. We first have some lemmas.
\begin{lem}\label{dimfil} With the notations above and assume that $R$ is catenary, then the filtration $\mathcal{D}_\fp :(D_{s_1})_\fp \subset (D_{s_2})_\fp \subset\ldots\subset (D_{s_l})_\fp =M_\fp $ is the dimension filtration of $M_\fp$.
\end{lem}
\begin{proof}
We need only to show that $(D_{s_{i-1}})_\fp$ is the biggest submodule of $(D_{s_i})_\fp$ with $\dim (D_{s_{i-1}})_\fp<\dim (D_{s_i})_\fp$ for the case $i=l$, the other cases are proved similarly. For an $R$-module $N$, denote $\Assh(N)=\{\fq\in \Ass(N): \dim R/\fq=\dim N\}$. Then it follows by Remark \ref{23}(i)  that $\Ass (D_{s_l}/D_{s_l-1})=\Assh (D_{s_l}/D_{s_l-1})$. Therefore,  by the choice of $s_{l-1}$ and the assumption $R$ is catenary, we get
 $\dim (D_{s_l}/D_{s_{l-1}})_\fp=\dim (D_{s_l}/D_{s_l-1})_\fp=d_{s_l}-\dim R/\fp$.  This implies that 
$$\dim (D_{s_{l-1}})_\fp\leqslant d_{s_{l-1}}-\dim R/\fp<\dim (D_{s_l}/D_{s_{l-1}})_\fp\leqslant \dim(D_{s_l})_\fp.$$ 
So $(D_{s_{l-1}})_\fp$ is the biggest submodule of $(D_{s_l})_\fp$ with $\dim (D_{s_{l-1}})_\fp<\dim (D_{s_l})_\fp$ by Remark \ref{23}(i) again.
\end{proof}

From now on, for a $\fp \in \Supp (M)$ we  always denote the dimension filtration of $M_\fp$ by $\mathcal{D}_\fp :(D_{s_1})_\fp \subset (D_{s_2})_\fp \subset\ldots\subset (D_{s_l})_\fp =M_\fp $, where $ s_1,\ldots , s_l$ are determined as in the beginning of this section. Belows are some basic properties of this  filtration.   
\begin{lem}\label{dfpro}
Let  $\underline{x}=x_1,\ldots,x_d$ be a good system of parameters  of $M$ and $\fp \in \mathrm{Supp}M$, $\fp \not=\frak{m}$. Assume that $R$ is catenary. Then
\item(i) $(D_j)_\fp =(D_{s_i})_\fp,$ for all $s_i\leqslant j<s_{i+1}$, $ i=1,\ldots,l.$
\item(ii) $(x_{d_{s_l}+1}, \ldots, x_d)\subseteq \mathrm{Ann}_{R_\fp}M_\fp$.
\item(iii) $\Assh(D_{s_i})_\fp =\{\fq R_\fp\: :\:\fq \in \Assh  D_{s_i},\: \fq \subseteq\fp \}$ and $$\dim D_{s_i}=\dim(D_{s_i})_\fp +\dim R/\fp ,$$ for $ i=1,\ldots,l$. 
\end{lem}
\begin{proof} ($i$)  is a consequence of the definition of the dimension filtration $\mathcal{D}_\fp$.
\item($ii$)  Since $\underline x$ is a good system of parameters, we have $D_{s_l}\cap(x_{d_{s_l}+1},\ldots,x_d)M=0$. This implies that 
$$M_\fp\cap(x_{d_{s_l}+1},\ldots,x_d)M_\fp=(D_{s_l})_\fp\cap(x_{d_{s_l}+1},\ldots,x_d)M_\fp=0.$$
Hence, $(x_{d_{s_l}+1},\ldots,x_d)\subseteq \mathrm{Ann}_{R_\fp}M_\fp$.
\item($iii$)  Since $\mathcal{D}_\fp$ is the dimension filtration of $M_\fp$ by Lemma \ref{dimfil}, we get $$\Assh(D_{s_i})_\fp=\Ass (D_{s_i}/D_{s_{i-1}})_\fp=\Ass (D_{s_i}/D_{s_i-1})_\fp.$$ Then the assertions can be shown as in the proof of Lemma \ref{dimfil}.
\end{proof}

\begin{pro}\label{localdf}
Let $R$ be a catenary ring and $\fp \in\mathrm{Supp}M$. There exists a good system  of parameters $\underline{x}=x_1,\ldots,x_d$ of $M$ such that $x_{r+1},\ldots,x_s$ is a good system of parameters of $M_\fp $, where   $r=\dim R/\fp $ and $s=\dim M_\fp+\dim R/\fp.$
\end{pro}   
\begin{proof}
Put $d_i=\dim D_i$. By Remark \ref{23}(i), $D_i=\bigcap_{\dim(R/\fp)\geq d_{i+1}}N(\fp)$ where $0=\bigcap_{\fp\in \Ass M}N(\fp)$ is a reduced primary decomposition in $M$. Put $N_i=\bigcap\limits_{\dim R/\fp\leqslant d_i}N(\fp)$. Then $D_i\cap N_i=0$ and $\dim (M/N_i)=d_i$. Since $r=\dim R/\fp $, by the Prime Avoidance Theorem there exists a system of parameters $\underline x=x_1,\ldots,x_d$ such that $x_{d_i+1},\ldots,x_d \in\mathrm{Ann}(M/N_i)$ for all $i=0, 1, \ldots, t$ and $(x_{r+1}, \ldots, x_d)\subseteq \fp$. We show that  $\underline x$ is the required system of parameters. In fact, since $(x_{d_i+1},\ldots,x_d)M\cap D_i\subseteq N_i\cap D_i=0$,  $\underline x$ is a good system of parameters of $M$. By Lemma \ref{dimfil}, the dimension filtration of $M_\fp$ is of the form
$$\mathcal{D}_\fp :(D_{s_1})_\fp \subset (D_{s_2})_\fp \subset\ldots\subset (D_{s_l})_\fp =M_\fp.$$
Since $r=\dim R/\fp$ and $(x_{r+1}, \ldots, x_d)\subseteq \fp$, the $R_\fp$-module $M_\fp/(x_{r+1}, \ldots, x_d)M_\fp$ is of finite length. On the other hand, by Lemma \ref{dfpro} we get $(x_{s+1},\ldots, x_d)M_\fp=0$ and  $s=\dim D_{s_l}$. Therefore $x_{r+1}, \ldots, x_s$ is a system of parameters of $M_\fp$. Since $\underline x$ is a good system of parameters, $(x_{d_{s_j}+1},\ldots,x_d)M\cap D_{s_j}=0$ for any $j=1, 2, \ldots, l$. Hence
$$(x_{d_{s_j}+1},\ldots,x_d)M_\fp\cap (D_{s_j})_\fp=0.$$
Thus  $x_{r+1}, \ldots, x_{d_{s_l}}$ is a good system of parameters of $M_\fp$ as required.
\end{proof}


\section{Proof of Theorem \ref{B}}
In \cite {C2} the first author had showed that the least degree of all polynomials in $n_1,\ldots , n_d$  bounding above the difference $\ell(M/\underline x(\underline n)M)-e(\underline x(\underline n);M)$ is independent of the choices of the system of parameters $\underline x$ and he denoted this invariant by $p(M)$ and  called it the polynomial type of $M$. In our circumstances, this polynomial type is just the invariant  $p_{\mathcal F_0}(M)$, where $\mathcal F_0: 0\subset M$ is 
 the trivial filtration of $M$. When $R$ is a quotient of a Cohen-Macaulay ring, it is well-known that the non-Cohen-Macaulay locus $\mathrm{nCM}(M)=\{\fp\in \Supp M : M_\fp \text{ is not Cohen-Macaulay}\}$ is closed. Then by Corollary 4.2 of \cite{C2}, we have $p(M)=\dim \mathrm{nCM}(M)$ provided $M$ is equidimensional.  To prove Theorem \ref{B} we need some auxiliary  lemmas.
\begin{lem}\label{nCM}
Let $\D: D_0\subset D_1\subset\ldots\subset D_t=M$ be the dimension filtration of $M$. Denote $V_M=\{\fp\in \Supp M: M_\fp$ is not sequentially Cohen-Macaulay$\}$. Then 
\item(i) $V_M=\bigcup\limits_{i=1}^t\mathrm{nCM}(D_i/D_{i-1}).$
\item(ii) If $R$ is a quotient of a Cohen-Macaulay ring then $V_M$ is closed.
\end{lem}

\begin{proof} ($i$)  is straightforward from Lemma \ref{dimfil}.\\
($ii$) If  $R$ is a quotient of a Cohen-Macaulay ring, $\mathrm{nCM}(D_i/D_{i-1})$ is closed for all $i=1,\ldots, t$. Hence $V_M$ is closed. 
\end{proof}
\begin{lem}\label{1}
Assume that $R$ is a quotient of a Cohen-Macaulay ring. Then
$$\dim V_M=\max\{p(D_i/D_{i-1}):i=1,\ldots,t\}.$$
\end{lem}
\begin{proof}
Since $D_i/D_{i-1}$ is equidimensional and $R$ is a quotient of a Cohen-\linebreak Macaulay ring, $\dim \mathrm{nCM}(D_i/D_{i-1})=p(D_i/D_{i-1})$ by \cite[Corollary 4.2]{C2}. So by Lemma \ref{nCM}, we have 
\[\begin{aligned}\dim V_M&=\max\{\dim \mathrm{nCM}(D_i/D_{i-1}):\forall i=1,\ldots,t\}\\
&=\max\{p(D_i/D_{i-1}):i=1,\ldots,t\}.
\end{aligned}\]
\end{proof}

\begin{lem} \label{2}
 Assume that $R$ is a quotient of a Cohen-Macaulay ring. Then $$p_\D(M)\leqslant \max\{p(D_i/D_{i-1}):i=1,\ldots,t\}.$$
\end{lem}
\begin{proof}
Let $\underline{x}=x_1, \ldots, x_d$ be a good system of parameters of $M$. We have $\ell(M/\underline xM)=\ell(M/\underline xM+D_{t-1})+\ell(\underline xM+D_{t-1}/\underline xM)\leqslant \ell(M/ \underline xM+D_{t-1})+\ell(D_{t-1}/\underline xD_{t-1})$. Put $d_i=\dim D_i$. Note that $x_1,\ldots,x_{d_i}$ is a good system of parameters of $D_i$ and $ \underline xD_i=(x_1,\ldots,x_{d_i})D_i$. By induction on $t$ we have
$$\ell(M/ \underline xM)\leqslant \sum\limits_{t=1}^t\ell\big(D_i/(x_1,\ldots,x_{d_i})D_i+D_{i-1}\big)+\ell(D_0).$$
Combine this and replace $\underline x$ by $ \underline x(\underline n)$, we obtain
$$I_{\mathcal{D},M}(\underline x(\underline n))\leqslant \sum\limits_{i=1}^t\big(\ell(D_i/(x_1^{n_1},\ldots,x_{d_i}^{n_{d_i}})D_i+D_{i-1})-e(x_1^{n_1},\ldots,x_{d_i}^{n_{d_i}}; D_i/D_{i-1})\big),$$ and the result follows.
\end{proof}

\begin{lem} \label{3}
Assume that $R$ is a quotient of a Cohen-Macaulay ring. Then 
$$\dim V_M\leqslant p_\mathcal D(M).$$
\end{lem}
\begin{proof}
We will prove that $M_\fp $ is a sequentially Cohen-Macaulay $R_\fp $-module for all prime ideals $\fp \in\mathrm{Supp}M$ such that $\dim R/\fp >p_\mathcal{D}(M)$.  By Lemmas \ref{dimfil} and Proposition \ref{localdf}, $M_\fp$ has the dimension filtration
$$\mathcal{D}_\fp :(D_{s_1})_\fp \subset (D_{s_2})_\fp \subset\ldots\subset (D_{s_l})_\fp =M_\fp,$$ and there is a good system of parameters $\underline{x}=x_1,\ldots,x_d$ of $M$ such that $x_{r+1}, \ldots, x_s$ is a good system of parameters of $M_\fp$,  where $r=\dim R/\fp $ and $s=\dim M_\fp+ \dim R/\fp$. First, we prove that $I_{\mathcal{D}_\fp ,M_\fp }(x_{r+1},\ldots,x_s )=0$. For all $i=1,\ldots,d,$  we set $d_i=\dim D_i$. By Lemma 2.4 of \cite{CC1}, we have $D_j=0:_Mx_i$ for all $d_j<i\leqslant d_{j+1}$. Hence
$e(x_1,\ldots, x_i; 0:_Mx_{i+1})=e(x_1,\ldots, x_i; D_j)$ if $i=d_j+1$ for some $j$ and $e(x_1,\ldots, x_i; 0:_Mx_{i+1})=0$ otherwise. Moreover, 
$$0:_Mx_{i+1}\simeq (0:_Mx_{i+1}+(x_{i+2}, \ldots, x_d)M)/(x_{i+2}, \ldots, x_d)M.$$
Then, by using Corollary 4.3 of \cite{AB}, we get
\[\begin{aligned}
I_{\mathcal{D},M}(\underline x)&=\ell(M/\underline xM)-\sum\limits_{i=0}^te(x_1,\ldots,x_{d_i};D_i)\\
&=\sum\limits_{i=0}^{d-1}e\big(x_1,\ldots, x_i;(x_{i+2},\ldots,x_{d})M:_Mx_{i+1}/(x_{i+2},\ldots,x_d)M\big)\\
&\hspace{0.5cm}-\sum\limits_{i=0}^{d-1}e\big(x_1,\ldots, x_i;0:_Mx_{i+1})\\
&=\sum\limits_{i=0}^{d-1}e\big(x_1,\ldots, x_i;(x_{i+2},\ldots,x_{d})M:_Mx_{i+1}/(x_{i+2},\ldots,x_d)M+0:_Mx_{i+1}\big)\\
&\ge e\big(x_1,\ldots, x_i;(x_{i+2},\ldots,x_{d})M:_Mx_{i+1}/(x_{i+2},\ldots,x_d)M+0:_Mx_{i+1}\big),
\end{aligned}\]
Replacing $\un x$ by $x_1^{n_1}, \ldots, x_i^{n_i}, x_{i+1}, \ldots, x_d$ we obtain
\begin{multline*}I_{\mathcal{D},M}(x_1^{n_1}, \ldots, x_i^{n_i}, x_{i+1}, \ldots, x_d)\geq\\ n_1\ldots n_ie\big(x_1,\ldots, x_i;(x_{i+2},\ldots,x_{d})M:_Mx_{i+1}/(x_{i+2},\ldots,x_d)M+0:_Mx_{i+1}\big),
\end{multline*}
for  all positive integers $n_1,\ldots,n_i,\ i=1,\ldots , t$. Note that $p_{\mathcal{D}}(M)$ is the degree of a polynomial bounding above the function $I_{\mathcal{D},M}(x_1^{n_1}, \ldots, x_i^{n_i}, x_{i+1}, \ldots, x_d)$. Therefore  
$$e\big(x_1,\ldots, x_i; (x_{i+2},\ldots,x_d)M:_Mx_{i+1}/(x_{i+2},\ldots,x_d)M+0:_Mx_{i+1}\big)=0$$ for all $i>p_\mathcal{D}(M)$. Let $i\geqslant r > p_D(M)$. By Proposition 4.7 of \cite{AB} we have $x_{i+1}\not\in \fq $ for all $\fq \in\Ass (M/(x_{i+2},\ldots,x_d)M+0:_Mx_{i+1})$ such that $\dim R/\fq \ge i$. 
 If $\dim R_\fp /\fq R_\fp \ge i-r$ then  $\dim R/\fq \ge i$, since $R$ is catenary. Hence $x_{i+1}\not\in \fq R_\fp $ for all  $\fq R_\fp \in\Ass \big(M_\fp /(x_{i+2},\ldots,x_s)M_\fp+0:_{M_\fp}x_{i+1}\big)$ such that $\dim R_\fp /\fq R_\fp \ge i-r$. Using Proposition 4.7 of \cite{AB} again we have 
$$e(x_{r+1},\ldots,x_i;(x_{i+2},\ldots,x_s)M_\fp :_{M_\fp }x_{i+1}/(x_{i+2},\ldots,x_s)M_\fp+0:_{M_\fp}x_{i+1})=0,$$
for all $i> r$. Therefore, 
\[\begin{aligned}
I&_{\mathcal D_\fp,M_\fp }(x_{r+1},\ldots,x_s)\\
&=\sum\limits_{i=r}^{s-1}e(x_{r+1},\ldots,x_i;(x_{i+2},\ldots,x_s)M_\fp :_{M_\fp }x_{i+1}/(x_{i+2},\ldots,x_s)M_\fp+0:_{M_\fp}x_{i+1})\\
&=0.
\end{aligned}\]
Finally, by replacing $x_{r+1},\ldots,x_s$ with $x_{r+1}^{n_{r+1}},\ldots,x_s^{n_s}$  for any positive integers $n_{r+1},\ldots , n_s$
 we can prove by the same method that  $I_{\mathcal{D}_\fp ,M_\fp }(x_{r+1}^{n_{r+1}},\ldots,x_s^{n_s})=0$. Therefore  $M_\fp $ is a sequentially Cohen-Macaulay $R_\fp $-module by virtue of Corollary 3.5.
\end{proof}
\vspace{.3cm}

\noindent{\bf Proof of Theorem \ref{B}.} Theorem \ref{B} follows  immediately from Lemma \ref{1}, Lemma \ref{2} and Lemma \ref{3}.
\vspace{.3cm}

In \cite{C2} the invariant $p(M)$ is studied in relations with the non-Cohen-Macaulay locus, the annihilators of the local cohomology modules and some others. Combining these results and Theorem \ref{B} we have the following immediate corollaries.
\begin{cor}
Let $R$ be a quotient of a Cohen-Macaulay ring and $\D: D_0\subset D_1\subset\ldots\subset D_t=M$ the dimension filtration of $M$. Let $k$ be an integer. The following are equivalent:
\item(i) $p_\D(M)\leqslant k$.
\item(ii) For all $\fp\in \Spec R$, $\dim R/\fp>k$, $M_\fp$ is sequentially Cohen-Macaulay and for each $i=1, \ldots, t$, either $\fp\not\in \Supp(D_i/D_{i-1})$ or $\dim R/\fp+\dim (D_i/D_{i-1})_\fp\geqslant \dim D_i$.
\end{cor}
\begin{proof}
The corollary is implied from Theorem \ref{B} and Theorem 4.1 of \cite{C2}.
\end{proof}
\begin{cor}
Let $R$ be a quotient of a Cohen-Macaulay ring and $\D: D_0\subset D_1\subset\ldots\subset D_t=M$ be the dimension filtration of $M$. Let  $\fp\in \Supp M$ with $\dim R/\fp \leqslant p_\D(M)$. Denote the dimension filtration of $M_\fp$ by $\D_\fp$. We have
$$p_{\D_\fp}(M_\fp)\leqslant p_\D(M)-\dim R/\fp.$$
\end{cor}
\begin{proof}
By Lemma \ref{dimfil}, there are $0\leqslant s_1<s_2\ldots<s_l\leqslant t$ such that the filtration $\mathcal{D}_\fp :(D_{s_1})_\fp \subset (D_{s_2})_\fp \subset\ldots\subset (D_{s_l})_\fp =M_\fp $ is the dimension filtration of $M_\fp$. We have by Theorem \ref{B},

\end{proof}

The last example shows that in Theorem \ref{B} we can not replace the dimension filtration $\D$ by a general filtration satisfying the dimension condition.

\begin{ex} Let $R=k[[X_1, X_2, X_3]]$ be the local ring of all formal power series with coefficients in a field $k$. Let $I=(X_1X_3, X_1X_4, X_1X_5, X_2X_3, X_2X_4, X_2X_5)$ and $M=R/I$. Put $M_1=(X_1X_2, X_2^2)+I/I\subset M$. Then $\dim M_1=2<\dim M$ and the filtration $\F: 0=M_0\subset M_1\subset M_2=M$ satisfies the dimension condition. It is easy to see that $x_1=X_1+X_4$, $x_2=X_2+X_5$, $x_3=X_3$ is a good system of parameters of $M$ with respect to $\F$. By computing directly we obtain
$$I_{\F, M}(x_1^{n_1}, x_2^{n_2}, x_3^{n_3})=1,\ \text{ for all } n_1, n_2, n_3>0.$$
Consequently, $p_\F(M)=0$. On the other hand, 
$$\ell(M/M_1+(x_1^{n_1}, x_2^{n_2}, x_3^{n_3})M)=n_1n_2n_3+n_1+1,\ \text{ for all } n_1, n_2, n_3>0.$$
Thus $p(M/M_1)=1$ and $p_\F(M)<p(M/M_1)$.
\end{ex}


\section{Flat extensions}

In this final section we study the behavior of the invariant $p_\F(M)$ under flat extensions. Let $(R,\fm)\to(S,\fn)$ be a local flat homorphism of catenary Noetherian local rings. Let $M$ be a finitely generated $R$-module and $\F: M_0\subset M_1\subset \ldots \subset M_t=M$ a filtration satisfying the dimension condition. Since $S$ is a flat extension of $R$, there corresponds to $\F$ a filtration of submodules of $M\otimes_R S$
$$\F\otimes S:M_0\otimes_RS\subset M_1\otimes_RS\subset\ldots\subset M_t\otimes_RS=M \otimes_RS.$$
Denote $l=\dim S/\fm S$ and $d_i=\dim M_i$, $i=0, \ldots, t$. By \cite[Theorem 15.1]{M}, if $R, S$ are catenary then $\dim M\otimes_RS=d+l$ and $\dim M_i\otimes_RS=d_i+l$ for $i=0, \ldots, t-1$. Thus the filtration $\F\otimes S$ satisfies the dimension condition and we obtain the invariant $p_{\F\otimes S}(M\otimes_RS)$ by Theorem \ref{A}. Keep these notations, we have the following lemma before stating the result relating $p_\F(M)$ and $p_{\F\otimes S}(M\otimes_RS)$.
\begin{lem}\label{51}
Let $(R,\fm)\to(S,\fn)$ be a local flat homomorphism of catenary Noetherian local rings. Assume that $M$ is a finitely generated $R$-module and $\F$ is a filtration of submodules of $M$ satisfying the dimension condition. The module $M\otimes_RS$ has a good system of parameters $x_1, \ldots, x_{l+d}$ with respect to the filtration $\F\otimes_RS$ such that $x_{l+1},\ldots,x_{l+d}$ is the image of a good system of parameter of $M$.
\end{lem}
\begin{proof}
Following \cite[Lemma 2.5]{CC1}, there always exists a good system of parameters $x_{l+1}, \ldots, x_{l+d}$ of $M$. Under the flat extension, we obtain a system of parameters $x_1, \ldots, x_l, x_{l+1}, \ldots, x_{l+d}\in \fn$ of $M\otimes_RS$. Moreover, we have
$$(x_{l+d_i+1}, \ldots, x_{l+d})(M\otimes_RS)\cap(M_i\otimes_RS)=((x_{l+d_i+1}, \ldots, x_{l+d})M\cap M_i)\otimes_RS=0.$$
So $x_1, \ldots, x_{l+d}$ is a good system of parameters of $M\otimes_RS$ with respect to $\F\otimes_RS$.
\end{proof}

\begin{theorem}\label{flat}
Let $\varphi:(R,\fm)\to (S,\fn)$ be a local flat homomorphism of catenary Noetherian local rings and $M$ a finitely generated $R$-module. Let $\F: M_0\subset M_1\subset \ldots\subset M_t=M$ be a filtration satisfying the dimension condition. Then
$$p_{\F\otimes_RS}(M\otimes_RS)=\max\{\dim S/\fm S+p_\F(M),\dim M+p(S/\fm S)\}.$$ 
\end{theorem}
\begin{proof}
Put $l=\dim S/\fm S$. Following Lemma \ref{51}, $M\otimes_RS$ has a good system of parameters $x_1, \ldots, x_{l+d}$ with respect to $\F\otimes S$ such that $x_{l+1}, \ldots, x_{l+d}$ is a good system of parameters of $M$. Put $d_i=\dim M_i$, $i=0, 1, \ldots, t$. We have
$$\frac{S}{(x_1,\ldots,x_l)+\fm S}\otimes_R\frac{M_i}{(x_{l+1},\ldots,x_{l+d_i})M_i}\cong \frac{M_i\otimes_RS}{(x_{1},\ldots,x_{l+d_i})M_i\otimes_RS}$$
Then,
$$\ell(\frac{M_i\otimes_RS}{(x_1,\ldots,x_{l+d_i})M_i\otimes_RS})=\ell(\frac{S}{(x_1,\ldots,x_l)+\fm S})\ell(\frac{M_i}{(x_{l+1},\ldots,x_{l+d_i})M_i}).$$
By Lech's formula we obtain
$$e(x_1,\ldots,x_{l+d_i};M_i\otimes_RS)=e(x_1,\ldots,x_l;S/\fm S)e(x_{l+1},\ldots,x_{l+d_i};M_i).
$$
Hence we have
$$\begin{aligned}
I_{\F\otimes_RS, M\otimes_RS}(x_1,\ldots,x_{l+d})
=&\ell(\frac{M\otimes_RS}{(x_{1},\ldots,x_{l+d})M\otimes_RS})-\sum\limits_{i=0}^te(x_1,\ldots,x_{d_i+l};M_i\otimes_RS)\\
=&\ell(\frac{S}{(x_{1},\ldots,x_{l})+\fm S})\ell(\frac{M}{(x_{l+1},\ldots,x_{l+d})M})\\
&-\sum\limits_{i=0}^te(x_1,\ldots,x_{l};S/\fm S)e(x_{l+1},\ldots,x_{d_i+l};S/\fm S)\\
=&\ell(\frac{M}{(x_{l+1},\ldots,x_{l+d})M})I_{S/\fm S}(x_1,\ldots,x_l)\\&+e(x_1,\ldots,x_{l};S/\fm S)I_{\F,M}(x_{l+1},\ldots,x_{l+d}),
\end{aligned}$$
where $I_{S/\fm S}(x_1,\ldots,x_l)=\ell(\frac{S}{(x_{1},\ldots,x_{l}) +\fm S})-e(x_1,\ldots,x_{l};S/\fm S)$. Therefore by Thereom \ref{A} we have
$$p_{\F\otimes_RS}(M\otimes_RS)=\max\{\dim S/\fm S+p_\F(M),\dim M+p(S/\fm S)\}.$$
\end{proof}

Therem \ref{flat} leads to some interesting consequences.

\begin{cor} \label{co63}
Keep all hypotheses as in Theorem \ref{flat}. Assume that $M$ has the dimension filtration $\D$.
\item(i) If $S/\fm S$ is Cohen-Macaulay then $p_{\F\otimes_RS}(M\otimes_RS)=\dim S/\fm S+p_\F(M)$ for any filtration $\F$ satisfying the dimension condition. 
\item(ii) If $M$ is sequentially Cohen-Macaulay then $p_{\D\otimes_RS}(M\otimes_RS)=\dim M+p(S/\fm S)$.
\item(iii) $M\otimes_RS$ is sequentially Cohen-Macaulay if $M$ is sequentially Cohen-Macaulay and $S/\fm S$ is Cohen-Macaulay.
\end{cor}
In fact, a module is Cohen-Macaulay if and only if so is its $\fm$-adic completion. However, it is not the case of the sequentially Cohen-Macaulay property. The reason is under a flat base change $R\rightarrow S$, the dimension filtration of $M$ is not preserved as the dimension filtration of $M\otimes_RS$. Let $(R, \fm)$ be the two-dimensional domain  considered by Ferrand-Raynaud in \cite{FR}. It is obvious that $R$ is not sequentially Cohen-Macaulay. However, the $\fm$-adic completion $\hat R$ has the dimension filtration $0=D_0\subset D_1\subset D_2=\hat R$ where $\dim D_1=1$ and $\hat R/D_1$ is Cohen-Macaulay by \cite[Example 6.1]{Sch}. Thus $\hat R$ is a sequentially Cohen-Macaulay ring. This shows that the converse of Corollary \ref{co63}($iii$) does not hold.

The next corollary of Theorem \ref{flat} provides a sufficient condition for the sequentially Cohen-Macaulay property on a module and its completion.
\begin{cor}
Let $\hat R$ be the $\fm$-adic completion of $R$ and $M$ a finitely generated $R$-module. Assume in addition that $R$ is catenary. Then  $p_{\D\otimes\hat R}(\hat M)=p_\D(M)$. In particular, if $M$ is sequentially Cohen-Macaulay then so is $\hat M$.
\end{cor}


\end{document}